\documentclass[11pt]{article}
\usepackage{amsmath,amsthm,amssymb,amsxtra, graphicx}
\usepackage{amsmath}
\usepackage{graphicx}
\usepackage{amsfonts}
\usepackage{amssymb}
\usepackage{color}

\date{}
\newtheorem{theorem}{Theorem}[section]
\newtheorem{lemma}{Lemma}[section]
\newtheorem{proposition}{Proposition}[section]
\newtheorem{corollary}{Corollary}[section]

\title{\bf  $L^{p}$-norms and Mahler's
measure of polynomials \\on the $n$-dimensional torus}
\author{Andreas Defant \,and\, Mieczys{\l}aw Masty{\l}o}
\begin{document}
\maketitle
\renewcommand{\thefootnote}{\fnsymbol{footnote}}

\footnotetext{The work was supported by The Foundation for Polish
Science (FNP).}

\maketitle

\vspace{5 mm}

\begin{abstract}
\noindent We prove Nikol'skii type inequalities which for
polynomials on the $n$-dimensional torus $\mathbb{T}^n$ \linebreak
relate the $L^p$-with the $L^q$-norm (with respect to the
normalized Lebesgue measure and $0 <p <q < \infty$).~Among other
things we show that $C=\sqrt{q/p}$ is the best constant such that
$\|P\|_{L^q}\leq C^{\text{deg}(P)} \|P\|_{L^p}$ for all
homogeneous polynomials $P$ on $\mathbb{T}^n$. We also prove an
exact inequality between the $L^p$-norm of a polynomial $P$ on
$\mathbb{T}^n$ and its Mahler measure $M(P)$, which is the
geometric mean of $|P|$ with respect to the normalized Lebesgue
measure on $\mathbb{T}^n$.~Using extrapolation we transfer this
estimate into a~Khintchine-Kahane type inequality, which, for
polynomials on $\mathbb{T}^n$, relates a~certain exponential
Orlicz norm and Mahler's measure. Applications are given,
including some interpolation estimates.
\end{abstract}

\vspace{3 mm}

\section{Introduction}
The Mahler measure of a~polynomial $P$
on $\mathbb{C}^n$ is given by
\begin{align} \label{limit}
M(P):= \exp \int_{\mathbb{T}^n} \log |P|\,dz =\lim_{p\to 0+}
\bigg(\int_{\mathbb{T}^n} |P|^p\,dz \bigg)^{1/p}\,,
\end{align}
where $dz = dz_1 \ldots dz_n$ stands for the normalized Lebesgue
measure on the $n$-torus $\mathbb{T}^n$. Thus $M(P)$ is the
geometric mean of $P$ over the $n$-torus $\mathbb{T}^n$ (we define
$M(0)=0$). We point out that Mahler \cite{Mahler0} used the
functional $M$ as a~powerful tool in a~simple proof of the
``Gelfond-Mahler inequality", which found important applications
in transcendence theory. It seems that the Mahler measure for
polynomials in one complex variable appears the first time in
Lehmer \cite{Lehmer}, where it is proved that, if
$P(z)=\sum_{k=0}^m a_k z^k\,, \, z \in \mathbb{C}$ with  $a_m\neq
0$  and zeros $\alpha_1, \ldots, \alpha_m \in \mathbb{C}$, then
\begin{align} \label{Lehmer}
M(P)= |a_m|\prod_{|\alpha_i| \ge 1} |\alpha_i|\,.
\end{align}
Let us also recall the following crucial multiplicativity property
of the Mahler measure $M$: For two polynomials $P$ and $Q$ on
$\mathbb{C}^n$,
\begin{align}
\label{multiplicativity} M(PQ) = M(P) M(Q)\,.
\end{align}
The following inequality due to  Arestov \cite{Arestov} is a~key
result for this article (first results in this direction are due
Mahler \cite{Mahler0} and Duncan
\cite{Duncan}): For every polynomial $P$ in one complex
variable and of degree $\text{deg}(P)\leq m$, and every $0 < p <
\infty$, we have
\begin{align}\label{Arestov}
\|P\|_{L^p(\mathbb{T})}  \leq \Lambda(p, m)\,M(P)\,,
\end{align}
where
\begin{align}\label{Arestov-con}
\Lambda(p, m) := \Big(\int_{\mathbb{T}} |1+z|^{mp}\,
dz)\Big)^{1/p}=2^{m} \pi^{-1/2p} \bigg(\frac{\Gamma(2^{-1}(mp +
1)}{\Gamma(2^{-1}(mp + 2)}\bigg)^{1/p}.
\end{align}
By definition, we have that $\|P\|_{L^p(\mathbb{T})} = \Lambda(p,
m)$ for the polynomial $P(z)= (1+z)^m\,, \, z \in \mathbb{C}$, and
moreover, by \eqref{Lehmer}, that $M(P)=1$, which implies that
\eqref{Arestov} is sharp. Arestov  proved in \cite{Arestov}, using
asymptotic formulas for the $\Gamma$-function, that for fixed
$p>0$,
\begin{align*}
\Lambda(p, m) = \bigg(\frac{2}{\pi p}\bigg)^{1/2p} 2^m
\big(m^{-1/2p} + o(m^{-1/2p})\big), \quad\, \text{as\, $m\to
\infty$}.
\end{align*}
Throughout the paper we will use standard notation; some
non-standard notation will be given locally.~Let
$\mathcal{P}(\mathbb{K}^n)$ be the space of all polynomials $P$ on
$\mathbb{K}^n$, where $\mathbb{K}=\mathbb{R}$ or
$\mathbb{K}=\mathbb{C}$, i.e.,
\[
P(z)=\sum_{\alpha} c_{\alpha}(P) z^{\alpha}\,,\quad
z=(z_1,...,z_n)\in \mathbb{K}^n\,.
\]
Here the sum is taken  over finitely many multi indices $\alpha
=(\alpha_1, \ldots, \alpha_n) \in \mathbb{N}_0^n$  and $z^{\alpha}
:= z_1^{\alpha_1}\,...\,z_n^{\alpha_n}$  denotes  the $\alpha$th
monomial. As usual, $ |\alpha|= \sum_{j=1}^n \alpha_j\,, $ and we
call
\begin{align}  \label{degreeone}
\text{deg}(P) := \max \{|\alpha|\,;\, c_\alpha(P)\neq 0\}
\end{align}
the total degree of $P$. If $m=\text{deg}(P)$ and all monomial
coefficients $c_\alpha=c_\alpha(P) =0 $ for all $|\alpha| < m$,
then $P$ is said to be $m$-homogeneous. For every polynomial $P\in
\mathcal{P}(\mathbb{K}^n)$ and $(u_1,...,u_n) \in \mathbb{K}^n$
the degree of the one variable polynomial
\begin{align} \label{1variable}
P_j(\cdot)=P(u_1, \ldots, u_{j-1},\,\cdot\,, u_{j+1}, \ldots, u_{n})
\end{align}
equals
\begin{align}\label{degree}
\text{deg}(P_j):=\max \big\{\alpha_j\,; \, c_\alpha(P) \neq 0
\big\} \,.
\end{align}
Finally, in contrast to \eqref{degreeone}, we write
\begin{align}  \label{degreeseveral}
\text{deg}_\infty(P) := \max \{\text{deg}(P_j)\,;\, 1 \leq j \leq
n\}\,.
\end{align}
The following  extension of \eqref{Arestov}, a sharp
Khintchine-Kahane type inequality which allows one to estimate the
$L^p(\mathbb{T}^n)$-norm $\|P\|_{L^p(\mathbb{T}^n)}$ of $P\in
\mathcal{P}(\mathbb{C}^n)$ by its  Mahler measure $M(P)$, was the
starting point of this article, and will be proved in Section
\ref{Mahler}: For every $P\in \mathcal{P}(\mathbb{C}^n)$ and every
$0 < p < \infty$,
\begin{equation*}
\|P\|_{L^p(\mathbb{T}^n)} \leq \bigg(\prod_{j=1}^{n} \Lambda \big(p,
\text{deg}(P_j)\big)\bigg) M(P)\,.
\end{equation*}
In view of \eqref{limit} this inequality can be seen as a limit
case of Khintchine-Kahane type inequalities relating for $0 <p <q
< \infty$ the $L^q{(\mathbb{T}^n)}$-norm of $P\in
\mathcal{P}(\mathbb{C}^n)$ by its $L^p{(\mathbb{T}^n)}$-norm. The
study of $L^p$-norms of polynomials has a~rich history, and in
Section \ref{qversusp} we give new information in this direction.
Recently, there has been a~considerable  interest in the behaviour
of the constants of Khintchine-Kahane type equivalences
$\|\cdot\|_{L^p} \approx \|\cdot\|_{L^q}$  in the case when
$L^p$-spaces are considered on arbitrary unit-volume convex bodies
$K$ in $\mathbb{R}^n$; see e.g. the work of Gromov-Milman
\cite{GromMil} and Bourgain \cite{Bour}. We also mention
\cite{Bobkov}, where it is proved that if $\mu$ is an arbitrary
log-concave probability measure $\mu$ on $\mathbb{R}^n$, then, for
all $1 \leq p < \infty$ and all $P\in \mathcal{P}(\mathbb{R}^n)$,
we have
\begin{align}\label{bobby}
\|P\|_{L^p(\mu)} \leq \left(p^{\frac{22}{\ln
2}}\right)^{\text{deg}_\infty(P)}\,\|P\|_{L^1(\mu)}\,.
\end{align}
A~Borel probability measure $\mu$ on $\mathbb{R}^n$ is said to be
\emph{log-concave} whenever $\mu$ is supported by some affine
subspace $E$, where it has a log-concave density $u\colon E\to [0,
\infty)$ (i.e., $\ln u$ is concave) with respect to the Lebesgue
measure on $E$. The inequality in \eqref{bobby}  implies that on
$\mathcal{P}(\mathbb{R}^n)$ all $L^p(\mu)$-norms are equivalent
with constants which are independent of $n$ and depend on
$\text{deg}_{\infty}(P)$ and $p$, only. Replacing $\mathbb{R}^n$
by $\mathbb{C}^n$, a well-known inequality due to Nikol'skii (see
\cite{Nikol1, Nikol2}) states that for every $P\in
\mathcal{P}(\mathbb{C}^n)$ and $0<q<p<\infty$,
\begin{align} \label{Niko}
\|P\|_{L^q(\mathbb{T}^n)} \leq
2^{n}\,\bigg(\prod_{j=1}^n\text{deg}(P_j)\bigg)^{1/q - 1/p}\,
\|P\|_{L^p(\mathbb{T}^n)}\,.
\end{align}
Such inequalities are called Nikol'skii inequalities or different
metrics inequalities, and there is an extensive literature on
them. (See, e.g., the monographs \cite{DeVLo} and \cite{ MMR}, and
\cite{Pierpont}, in which such inequalities in higher dimensions
over product domains in $\mathbb{C}^n$ are studied.) Based on
a~deep result of Weissler \cite{We80} on the hypercontractivity of
convolution with the Poisson kernel in Hardy spaces,  Bayart
proved in  \cite{Ba02} that for each homogeneous polynomial $P \in
\mathcal{P}(\mathbb{C}^n)$ and $0 < p <q < \infty$ we have,
\begin{align} \label{Bayart}
\|P\|_{L^q(\mathbb{T}^n)}\leq
\sqrt{\frac{q}{p}}^{{\rm{deg}}(P)}\,\,
\|P\|_{L^p(\mathbb{T}^n)}\,.
\end{align}
This Khintchine-Kahane type inequality extends earlier work  of
Beauzamy, Bombieri, Enflo, and Montgomery from \cite{BBEM} ($p=2<q
$ and $p< q=2$) and Bonami \cite{Bonami} ($q=2,p=1$). The striking
fact is that the constants involved in \eqref{Bayart} are
independent of the dimension $n$, while they grow exponentially
with the degree $m$. On the other hand, the constant in
\eqref{Niko} approach to infinity as $n$ tends to infinity.

In Section \ref{qversusp} we supply more information on
\eqref{Bayart} proving that for $0 < p < q  < \infty$ there is
a~constant $C=C(p,q)>0$  such that for every (not necessarily
homogeneous) $P \in \mathcal{P}(\mathbb{C}^n)$
\begin{align} \label{Sta}
\|P\|_{L^q(\mathbb{T}^n)}\leq C^{{\rm{deg}}(P)}\,\,
\|P\|_{L^p(\mathbb{T}^n)}\,.
\end{align}
Moreover we show, using a~result of Kwapie\'n (here published with
his permission), that the constant $C=\sqrt{\frac{q}{p}}$ is best
possible if in \eqref{Sta} we consider only all homogeneous
polynomials instead of all polynomials. Several interesting
applications are in order -- some motivated by the original work
of Mahler and his followers.

\vspace{2 mm}

\section{$L^p$-norms versus $L^q$-norms of polynomials on $\mathbb{T}^n$}
\label{qversusp}

The following  Khintchine-Kahane type inequality is the main
result of this section.

\begin{theorem}
\label{CORRO}
Let $0 < p < q  < \infty$.
\begin{itemize}
\item[{\rm(i)}] There is a constant $ \sqrt{\frac{q}{p}} \leq C
\leq \sqrt{\frac{q}{\min\{p,2\}}}$ such that for each $n$ and
every $P \in \mathcal{P}(\mathbb{C}^n)$ we have
\begin{equation}\label{homo2}
\|P\|_{L^q(\mathbb{T}^n)} \leq
C^{{\rm{deg}}(P)}\,\, \|P\|_{L^p(\mathbb{T}^n)}\,.
\end{equation}
\item[{\rm(ii)}] The best possible constant $C=C(p,q)$ such that
\eqref{homo2} holds for all homogeneous polynomials equals
$\sqrt{\frac{q}{p}}$.
\end{itemize}
\end{theorem}

\vspace{2 mm}

As mentioned in the discussion preceeding inequality
\eqref{Bayart}, the upper estimate ${\rm{(ii)}}$ is due to Bayart
\cite{Ba02}; see also \cite[Lemma 1.C1]{BBEM}
 and  \cite[Theorem
III.7]{Bonami} ($q=2,p=1$ with $C= \sqrt{2}$).

\vspace{1.5 mm}

\begin{proof}[Proof of Statement ${\rm(i)}$ in Theorem $\ref{CORRO}$]
\noindent Recall the following important result from Weissler's
article \cite[Corollary 2.1]{We80}: Given $0<p<q < \infty$, the
constant $r=\sqrt{\frac{p}{q}}$ is the best constant $0<r\leq 1$
such that for every  $P\in \mathcal{P}(\mathbb{C})$,
\begin{equation} \label{one}
 \left(\int_\mathbb{T} \big| P(rz)\big|^q\,dz
\right)^{1/q}\leq \left(\int_\mathbb{T} \big|
P(z)\big|^p\,dz\right)^{1/p}\,.
\end{equation}
In \cite[Theorem 9]{Ba02} Bayart used a~standard iteration
argument through Fubini's theorem in order to extend \eqref{one}
to polynomials in several variables: The constant
$r=\sqrt{\frac{p}{q}}$ is the best  constant $0 \leq r\leq 1$ such
that for every $P\in \mathcal{P}(\mathbb{C}^n)$ we have that
\begin{equation}
\label{several}
\left(\int_{\mathbb{T}^n} \big| P(rz)\big|^q\,dz\right)^{1/q} \leq
\left(\int_{\mathbb{T}^n} \big| P(z)\big|^p \,dz\right)^{1/p}\,.
\end{equation}
Clearly the optimality in \eqref{several} follows from the
optimality in \eqref{one}. Now we give the proof of  statement (i)
in three cases. Case $0 < p <q \leq 2$: From \eqref{several} we
deduce for $C = \sqrt{\frac{q}{p}}$ that
\begin{align*}
\|P\|_{L^q(\mathbb{T}^n)} & \leq \|P\|_{L^2(\mathbb{T}^n)} =
\left(\sum_\alpha \big| c_\alpha(P) \big|^2\right)^{1/2} \\
& \leq \max \big\{C^{|\alpha|}; \, c_\alpha(P) \neq 0 \big\}
\left(\sum_\alpha \bigg| \frac{c_\alpha(P)}{C^{|\alpha|}}
\bigg|^2\right)^{1/2} \\
& = \, C^{\text{deg}(P)}\, \left(\int_{\mathbb{T}^n}
\big|\sum_{\alpha} c_\alpha(P) \Big(\frac{z}{C}\Big)^\alpha
\big|^2\,dz\right)^{1/2} \leq \, C^{\text{deg}(P)}\,
\|P\|_{L^p(\mathbb{T}^n)}.
\end{align*}
\noindent Case $2 \leq p < q < \infty$: Define $C =
\sqrt{\frac{q}{2}}$. Again with \eqref{several} we deduce
\begin{align*}
\|P\|_{L^q(\mathbb{T}^n)} &
\leq \left(\int_{\mathbb{T}^n} \big| P(Cz)\big|^2\,dz\right)^{1/2}
= \left(\sum_\alpha \big| C^{|\alpha|} c_\alpha(P)
\big|^2\right)^{1/2} \\
&\leq \max \big\{C^{|\alpha|};\, c_\alpha(P) \neq 0 \big\}
\left(\sum_\alpha \big| c_\alpha(P) \big|^2\right)^{1/2}
\\&
= \, C^{\text{deg}(P)}\, \|P\|_{L^2(\mathbb{T}^n)}
\leq \, C^{\text{deg}(P)}\,
\|P\|_{L^p(\mathbb{T}^n)}.
\end{align*}
The case $p \leq 2 \leq q$ is an obvious combination of the
previous two cases; this time we get the constant $C =
\sqrt{\frac{2}{p}}\sqrt{\frac{q}{2}} = \sqrt{\frac{q}{p}}$. This
proves (i) with the upper bound $C \leq
\sqrt{\frac{q}{\min\{p,2\}}}$ for the constant. The lower bound
for $C$ follows from (ii) which will be proved after the next
proposition.
\end{proof}

The proof of statement (ii) in Theorem \ref{CORRO}  will be
a~consequence of the following result due to Kwapie\'n (private
communication). For fixed $0 < p < q < \infty$ and $m \in
\mathbb{N}$ let $C=C(p,q;m)$  be the best constant such that
\[
\|P\|_{L^q(\mathbb{T}^n)} \leq
C^{{\rm{deg}}(P)}\,\, \|P\|_{L^p(\mathbb{T}^n)}
\]
holds for all $m$-homogeneous polynomials $P \in
\mathcal{P}(\mathbb{C}^n)$ that are affine in each variable (i.e.,
all $m$-homogeneous polynomials with $\text{deg}_\infty (P) =1$).

\vspace{2 mm}

\begin{proposition} \label{homo3}
For every $0 < p < q< \infty$ and each $m \in \mathbb{N}$,
\begin{equation}  \label{low}
\displaystyle
\frac{m^{\frac{1}{2q}}}{m^{\frac{1}{2p}}}
\,\,\sqrt{\frac{q}{p}}^m \,\,\asymp\,\,
 \frac{\Gamma(\frac{qm}{2}+1)^{\frac{1}{q}}}{\Gamma(\frac{pm}{2}+1)^{\frac{1}{p}}}
\,\leq \, \,C(p,q;m)^m \, \,\leq \, \, \sqrt{\frac{q}{p}}^m\,,
\end{equation}
where the symbol $\asymp$ means that the terms are equal up to
constants only depending on $p$ and $q$.
\end{proposition}

\noindent The proof of the upper estimate is clear by now, but the
proof of the lower bound in \eqref{low} needs a bit more
preparation: For each $k,n \in \mathbb{N}$ with $n \ge k$ define
the following two $k$-homogeneous polynomials on $\mathbb{C}^n$:

\begin{align*}
&
e_{k,n}(z) = \sum_{1\le i_1<i_2,..<i_k \le n}z_{i_1},...z_{i_k}
\\&
p_{k,n}(z) = \sum_{i=1}^n z_i^k\,;
\end{align*}
in the literature $e_{k,n}$ and $p_{k,n}$ are called the  $k$-th
\emph{elementary symmetric} and  the  $k$-th  {\it power
symmetric} polynomial, respectively. It is known  that
\begin{align}
\label{Newton}
k!e_{k,n} =  p_{1,n}^k +
w_k(p_{1,n},p_{2,n},...,p_{k,n})\,,
\end{align}
where
\begin{equation} \label{Newton1}
w_{k,n} =\sum _{j_1,j_2...,j_k}\,a_{j_1,\dots,j_k}\, \, p_{1,n}^{j_1}p_{2,n}^{j_2}...p_{k,n}^{j_k}\,
\end{equation}
and the sum extends over all  $j_1,...j_k \in \mathbb{N}_0$ with
$j_1+2j_2+...+kj_k = k$, and at least one of the indices $j_2,...
j_k $ is not zero. Note that there is a recursive formula for the
coefficients $a_{j_1,..,j_k}$ (seemingly  due to Newton).

\vspace{2 mm}

\begin{lemma}
For $m \in \mathbb{N}$ and $n \in \mathbb{N}$ define
the $m$-homogeneous polynomial
$$U_{m,n}(z) = m!\,e_{m,n}\left(\frac{z_1}{\sqrt{n}},...,\frac{z_n}{\sqrt{n}}\right),
\quad\, z \in \mathbb{T}^n\,.$$
Then for each $0 < p < \infty$ we have that
\[
\lim_{n \to \infty}  \big\| U_{m,n} \big\|_{L^p(\mathbb{T}^n)}  =
\Gamma\left(\frac{mp}{2}+1\right)^{\frac{1}{p}}\,.
\]
\end{lemma}

\vspace{1.5 mm}

\begin{proof}
For each $1\leq k \leq m$ let
$$Q_{k,n}(z)= p_{k,n}\left(\frac{z_1}{\sqrt{n}},...,\frac{z_n}{\sqrt{n}}\right).$$
We treat $(z_i)$ as the Steinhaus r.v. By the central limit
theorem the sequence $Q_{1,n}$ converges in distributions, as
$n\to \infty$, to a canonical complex gaussian r.v. $G$, and for
each $t>0$ there is $C_t >0$ such that for all $n$ we have
$$\mathbb{E}\big|Q_{1,n}\big|^t\le  C_t\,.$$
Hence
\[
\lim_{n \to \infty} \mathbb{E}\big|Q_{1,n}^k\big|^p =
\mathbb{E}|G|^{kp} =\Gamma\left(\frac{pk}{2}+1\right), \quad\,
1\leq k\leq m.
\]
Each  $Q_{k,n}$ is distributed the same as
$Q_{1,n}n^{-\frac{k-1}{2}}$ (because $z^k_i$ is distributed the
same as $z_i$). Therefore for each $k>1$ and $t>0$
$$\lim_{n\to \infty} \mathbb{E}\big|Q_{k,n}\big|^t = 0\,,$$
and hence by the Minkowski/H\"older inequality and \eqref{Newton1}
for each $k>1$ and $t>0$,
\begin{align*}
&\lim_{n\to \infty}
\left(\mathbb{E}\left|w_k\left(\frac{z_1}{\sqrt{n}},...\frac{z_n}{\sqrt{n}}\right)\right|^t\right)^{\frac{1}{t}}
\leq \sum _{j_1,j_2...,j_k}\,a_{j_1,\dots,j_k}\, \,  \lim_{n\to
\infty} \left( \mathbb{E}
\left|Q_{1,n}^{j_1}...Q_{k,n}^{j_1}\right|^t\right)^{\frac{1}{t}} \\
&\leq \sum _{j_1,j_2...,j_k}\,a_{j_1,\dots,j_k}\, \,  \lim_{n\to
\infty} \left(    \mathbb{E}  \left|    Q_{1,n}\right|^{j_1\cdot k
t}\right)^{\frac{1}{k  t}} \ldots \lim_{n\to \infty} \left(
\mathbb{E} \left|    Q_{k,n}\right|^{j_1\cdot k
t}\right)^{\frac{1}{k  t}} =0 \,.
\end{align*}
Finally, \eqref{Newton} gives  $$\lim_{n\to \infty}
\mathbb{E}\big|U_{m,n}\big|^p =
\Gamma\left(\frac{mp}{2}+1\right)\,,$$ the desired result.
\end{proof}

\vspace{1.5 mm}

\begin{proof}[Proof of Proposition $\ref{homo3}$] Recall that only the
lower estimate in \eqref{low} remains to be shown.~An immediate
consequence of the preceding lemma is that
\[\lim_{n\to \infty} \frac{\|U_{m,n}\|_q}{\|U_{m,n}\|_p} =
\frac{\Gamma\left(\frac{qm}{2}+1\right)^{\frac{1}{q}}}{\Gamma\left(\frac{pm}{2}+1\right)^{\frac{1}{p}}}\,\,,
\]
and hence we get, by the quantitative version of Stirling's
formula, ($\sqrt{2\pi} x^{x+ 1/2}e^{-x} <\Gamma(x+1) < \sqrt{2\pi}
x^{x + 1/2} e^{-x + 1/12x}$ for all $x>0$),
\[
\displaystyle C(p,q;m)^m \, \ge \,
\frac{\Gamma(\frac{qm}{2}+1)^{1/q}}{\Gamma(\frac{pm}{2}+1)^{1/p}}
\,\asymp\, \frac{(qm\pi)^{\frac{1}{2q}}}{(pm\pi)^{\frac{1}{2p}}}
\,\sqrt{\frac{q}{p}}^m\,\,,
\]
which is the desired estimate.
\end{proof}

We finish by completing the proof of Theorem $\ref{CORRO}$.

\vspace{2 mm}

\begin{proof}[Proof of statement ${\rm{(ii)}}$ in Theorem $\ref{CORRO}$]
Assume that \eqref{homo2} holds with the constant $C$ for all
homogeneous polynomials on $\mathbb{C}^n$. Using \eqref{low} we
see that there is some constant $D=D(p,q)$ such that for all $m$,
\begin{align*}
D^{\frac{1}{m}}\left(\frac{m^{\frac{1}{2q}}}{m^{\frac{1}{2p}}}\right)^{\frac{1}{m}}
\,\,\sqrt{\frac{q}{p}} \leq  C\,,
\end{align*}
and hence we get the desired result when $m$ tends to infinity.
\end{proof}

\vspace{2 mm}

\section{$L^p$-norms versus Mahler's measure for polynomials on $\mathbb{T}^n$} \label{Mahler}

Based on Arestov's estimate \eqref{Arestov} we prove an exact
inequality between the $L^p$-norm of a polynomial $P$ on
$\mathbb{T}^n$ and its Mahler measure $M(P)$.

\vspace{2 mm}

\begin{theorem}
\label{estimate-mahler}  For every $P\in
\mathcal{P}(\mathbb{C}^n)$ and every $0 < p <
\infty$
\begin{equation} \label{jetztaber}
\|P\|_{L^p(\mathbb{T}^n)} \leq \bigg(\prod_{j=1}^{n} \Lambda (p,
d_j)\bigg) M(P)\,,
\end{equation}
where $d_j={\rm{deg}}(P_j)$ is as in \eqref{degree}. Moreover,
this inequality is sharp since for the polynomial
\[
P(z_1,...,z_n) =  P_1(z_1)\cdots P_n(z_n), \quad\, (z_1,...,z_n)\,
\in \mathbb{C}^n
\]
with $P_j(z)=(1 + z)^{d_j}\,,\,\,z\in \mathbb{C}\,$  and $d_j \in
\mathbb{N}\,,\,\,1\leq j\leq n$\,,\, we have that
$\|P\|_{L^p(\mathbb{T}^n)} = \prod_{j=1}^{n} \Lambda (p, d_j)$ as
well as  $M(P)=1$.
\end{theorem}

\vspace{1.5 mm}

\begin{proof}
We use  induction with respect to $n$.~The inequality is true for
$n=1$ by Arestov's result from \eqref{Arestov}. Fix a~positive
integer $n\geq 2$ and assume that the inequality is true for
$n-1$.~We fix a~positive sequence $(q_k)$ such that
$0<q_k<\text{min}\{1, p\}$ for each $k\in \mathbb{N}$ and
$q_{k}\downarrow 0$, and put for $1\leq k\leq n$
\[
C_k=\prod_{j=1}^k \Lambda(p, d_j)\,.
\]
Combining Fubini's theorem with Jensen's inequality and the
definition \eqref{limit} gives, by the inductive hypothesis,
\begin{align*}
\|P\|_{L^p(\mathbb{T}^n)} & = \bigg(\int_{\mathbb{T}}
\bigg(\int_{\mathbb{T}^{n-1}} |P(z_1,...,z_{n-1},
z_n)|^p\, dz_1 \ldots dz_{n-1}\bigg) \,dz_n
\bigg)^{1/p}
\\
& \leq C_{n-1}
\,\bigg(\int_{\mathbb{T}}\bigg(\text{exp}\int_{\mathbb{T}^{n-1}}
\log |P(z_1,..., z_{n-1}, z_n)|\, dz_1 \ldots
dz_{n-1}\bigg)^p\, d z_n\bigg)^{1/p}
\\
& = C_{n-1}\,\bigg(\int_{\mathbb{T}} \lim_{k\to \infty}
\bigg(\int_{\mathbb{T}^{n-1}} |P(z_1,...,z_{n-1},
z_n)|^{pq_k}\, dz_1 \ldots
dz_{n-1}\bigg)^{p/pq_k}\,dz_n\bigg)^{1/p}\,.
\end{align*}
It now follows by Fatou's lemma and (the continuous) Minkowski's
inequality (we have $0 <pq_k <p$ for each $k\in \mathbb{N}$)
\begin{align*}
\|P\|_{L^p(\mathbb{T}^n)} & \leq C_{n-1}\,\liminf_{k\to \infty}
\bigg(\int_ {\mathbb{T}} \bigg(\int_{\mathbb{T}^{n-1}}
|P(z_1,...,z_{n-1}, z_{n})|^{pq_k}\, d z_1\ldots
dz_{n-1}\bigg)^{p/pq_k}\,d z_n\bigg)^{1/p}
\\
& \leq C_{n-1} \,\liminf_{k\to \infty}
\bigg(\int_{\mathbb{T}^{n-1}} \bigg(\int_{\mathbb{T}}
|P(z_1,...,z_{n-1},z_n)|^{p}\,d z_n\bigg)^{pq_{k}/p}\,
 dz_1 \ldots dz_{n-1} \bigg)^{1/pq_k}\,.
\end{align*}
Then by Arestov's estimate from \eqref{Arestov} we obtain that
\begin{align*}
\|P\|_{L^p(\mathbb{T}^n)} & \leq  C_{n-1}\,\Lambda(p, d_n)\\
& \times \liminf_{k\to \infty} \bigg(\int_{\mathbb{T}^{n-1}}
\bigg(\text{exp}\int_{\mathbb{T}} \log |P(z_1,...,z_{n-1},
z_n)|\,dz_n \bigg)^{pq_k}\, dz_1 \ldots
dz_{n-1}\bigg)^{1/pq_k}\,,
\end{align*}
which by another application of the inductive hypothesis and
Fubini's theorem finally leads to the desired estimate,
\begin{align*}
\|P\|_{L^p(\mathbb{T}^n)} &   \leq   C_n \,\text{exp}
\bigg(\int_{\mathbb{T}^{n-1}} \log \bigg(\text{exp}
\int_{\mathbb{T}} \log |P(z_1,...,z_{n-1},z_n)|\,d
z_n\bigg)\, dz_1 \ldots dz_{n-1}\bigg)
\\
& = C_n \,\text{exp}\bigg(\int_{\mathbb{T}^n} \log
|P(z_1,...,z_n)|\, dz_1 \ldots dz_n\bigg) =
C_n\,M(P).
\end{align*}
It remains to check the comment on the sharpness of the
inequality:  By Fubini's theorem and the definition of $\Lambda
(p,d_j)$ from \eqref{Arestov-con} we immediately see that
$\|P\|_{L^p(\mathbb{T}^n)} = \prod_{j=1}^{n} \Lambda (p, d_j)$. On
the other hand the multiplicativity property \eqref{multiplicativity} of the Mahler measure
 gives that  $M(P)= \prod_{j=1}^{n}
M(P_j) =1$, where the latter equality follows from \eqref{Lehmer}.
This completes the proof.
\end{proof}

We conclude this section with a simple corollary.

\vspace{2 mm}

\begin{corollary}
\label{m-polynomial}  For every $P\in \mathcal{P}(\mathbb{C}^n)$
with $m= {\rm{deg}}_\infty (P)$, and every $0 < p < \infty$ we
have
\[
\|P\|_{L^p(\mathbb{T}^n)} \leq \Lambda(p, m)^n \,M(P).
\]
\end{corollary}

\vspace{1.5 mm}

\begin{proof} Fix  $p>0$. Substituting the polynomial $P\equiv 1$
into \eqref{Arestov} yields $1 \leq \Lambda(p, m)$ for each $m\in
\mathbb{N}$. Since $\Lambda(m, p)$ is the least constant in
inequality \eqref{Arestov} and the class of polynomials becomes
larger as $m$ grows, it follows that $\Lambda(p, k_1) \leq
\Lambda(p, k_2)$ provided $k_1 < k_2$. Now if  $P\in
\mathcal{P}(\mathbb{C}^n)$ with $\text{deg}_\infty (P)=m$, then
$d_j := \text{deg}(P_j)\leq m$ for each $1\leq j\leq n$ and so
$\prod_{j=1}^{n} \Lambda (p,d_j) \leq \Lambda(p, m)^n$.  Thus the
required estimate follows from Theorem \ref{estimate-mahler}.
\end{proof}

\vspace{2 mm}

\section{Applications}

In the final section we discuss some applications of our previous
results. The first one is motivated by a~result due to Bourgain
\cite{Bour} (answering affirmatively a~question by Milman), which
states that there are universal constants $t_0>0$ and $c\in (0,
1)$ such that, for every convex set $K\subset \mathbb{R}^n$ of
volume one and every $P\in \mathcal{P}(\mathbb{R}^n)$, the
following distribution inequality holds
\begin{align} \label{Bour1}
\mu_K\{|P| \geq t\,\|f\|_1\} \leq \text{exp}(-t^{c/n}),\quad\,
t\geq t_0,
\end{align}
where $\mu_K$ is the Lebesgue measure on $K$, and $\|f\|_1$ is the
$L^1$-norm of $P$ with respect to $\mu_K$. It is known that such
inequalities  may be rewritten in terms of the Orlicz space
$L_{\psi_{\alpha}}(K)$ generated by the convex function
$\psi_{\alpha}(t)= \text{exp}(t^{\alpha}) - 1$, $t\geq 0$ as
follows:
\begin{align} \label{Bour2}
\|P\|_{L_{\psi_{\alpha}}(K)} \leq C^{n} \|P\|_1, \quad\, P\in
\mathcal{P}(\mathbb{R}^n)\,,
\end{align}
where $\alpha = c/n$  and $C>0$ is some absolute constant.

We recall that for a given convex function $\varphi\colon [0,
\infty) \to [0, \infty)$ with $\varphi^{-1}(\{0\})=\{0\}$ and any
measure space $(\Omega, \Sigma, \mu)$, the Orlicz space
$L_{\varphi}(\Omega)$ is defined to be the space of all complex
functions $f\in L^0(\mu)$ such that
$\int_{\Omega}\varphi(\lambda|f|)\,d\mu <\infty$ for some $\lambda
>0$, and it is equipped with the norm
\[
\|f\|_{L_{\varphi}(\Omega)} :=\inf\bigg\{\lambda>0;\,
\int_{\Omega}\varphi\bigg(\frac{|f|}{\lambda}\bigg)\,d\mu \leq
1\bigg\}\,.
\]

\vspace{1.5 mm}

Now, using an extrapolation trick, we will prove a~variant of
\eqref{Bour2}, but for polynomials on the $n$-torus $\mathbb{T}^n$
instead of $\mathbb{R}^n$.~We will use Corollary
\ref{estimate-mahler}  to deduce the following Khintchine-Kahane
type inequality relating, for polynomials on the $n$-torus,
a~corresponding exponential Orlicz norm and Mahler's measure.

\vspace{2 mm}

\begin{theorem}\label{Orlicz}
For every $P\in \mathcal{P}(\mathbb{C}^n)$ with $m =
{\rm{deg}}_\infty(P)$ we have
\[
\|P\|_{L_{\psi_{1/m}}(\mathbb{T}^n)} \leq (2^n (e-1))^m \,M(P)\,.
\]
\end{theorem}

\vspace{1.5 mm}

\begin{proof}
We claim that for each $k\in \mathbb{N}$,
\begin{align} \label{eq1}
\Lambda\Big(\frac{k}{m}, m\Big) <2^{m}\,.
\end{align}
Indeed, recall first that the function $\Gamma$ reaches its only
minimum on $(0, \infty)$ at $x_{\text{min}}\approx 1.4616321451$
(see, e.g., \cite[p.~303]{Pierpont}). In particular,
$x_{\text{min}} < (k+1)/2 $ for each $k\geq 2$. Since $\Gamma$ is
increasing  on $(x_{\text{min}}, \infty)$,  it follows
 that for  each integer
 $k\geq 2$,
\begin{align*}
\Lambda\Big(\frac{k}{m}, m\Big)  = \frac{2^m}{\sqrt{\pi}^{m/k}}
\Big(  \frac{\Gamma((k+1)/2)}{\Gamma((k+2)/2)} \Big)^{m/k} \leq
\frac{2^m}{\sqrt{\pi}^{m/k}} <2^{m}.
\end{align*}
Moreover, for $k=1$ we easily check
\begin{align*}
\Lambda\Big(\frac{1}{m}, m\Big)  = \frac{2^m}{\sqrt{\pi}^{m}}
\Big(\frac{\Gamma(1)}{\Gamma(3/2)} \Big)^{m} = 2^m\,
\Big(\frac{2}{\pi}\Big)^{m} <2^{m}\,.
\end{align*}
which proves \eqref{eq1}. Now fix $P\in
\mathcal{P}(\mathbb{C}^n)$. Combining the above inequality with
Corollary \ref{m-polynomial} (for $p=k/m$) yields,
\[
\int_{\mathbb{T}^n} |P|^{k/m}\,d\lambda_n \leq
\Lambda\Big(\frac{k}{m}, m\Big)^{\frac{nk}{m}}\,M(P)^{\frac{k}{m}}
\leq 2^{ nk  }\,M(P)^{\frac{k}{m}}, \quad\, k\in \mathbb{N}.
\]
Recall that  $\psi_{1/m}(t) = \text{exp}(t^{1/m})-1$ for all
$t\geq 0$. Then, using Taylor's expansion and \eqref{eq1}, we
obtain that
\begin{align*}
\int_{\mathbb{T}^n} \psi_{1/m}\bigg(\frac{|P|}{2^{nm} (e-1)^m
 M(P)}\bigg)\,dz
& = \sum_{k=1}^{\infty} \frac{1}{ k!}\,\int_{\mathbb{T}^n}
\frac{|P|^{k/m}}{2^{kn}(e-1)^kM(P)^{k/m}}\,dz \\
& \leq \frac{1}{e-1} \sum_{k=1}^{\infty} \frac{1}{
k!}\,\int_{\mathbb{T}^n}
\frac{|P|^{k/m}}{2^{kn}M(P)^{k/m}}\,dz \\
& \leq \frac{1}{e-1}\sum_{k=1}^{\infty} \frac{1}{k!} =1 \,.
\end{align*}
Hence, by the definition of the norm in the Orlicz space
$L_{\psi_{1/m}}(\mathbb{T}^n)$, we get
\[
\|P\|_{L_{\psi_{1/m}}(\mathbb{T}^n)}  \leq (2^n (e-1))^m \,M(P)
\]
and this completes the proof.
\end{proof}

\vspace{2 mm}

The final applications are motivated by some interesting results
from the remarkable article \cite{Mahler2}. To explain these
results, following \cite{Mahler2}, for a given polynomial $P \in
\mathcal{P}(\mathbb{C}^n)$, we define
\[
L(P) := \sum_\alpha |c_\alpha(P)| \,\,\, \text{and }\,\,\, H(P) :=
\max_\alpha |c_\alpha(P)| \,.
\]
In \cite{Mahler2} Mahler established a~number of inequalities
connecting $L(P)$, $H(P)$ and $M(P)$, and showed applications to
estimates of
$\|P\|_{L^2(\mathbb{T}^n)}$ in terms of $L(P)$, $H(P)$, or $M(P)$.
A series of papers studies this and related problems. In
particular, the problem of finding norm estimates
\begin{equation}
\label{product} \|PQ\| \geq C \|P\|\,\|Q\|,
\end{equation}
where  $\|\cdot\|$ is some norm on $\mathcal{P}(\mathbb{C}^n)$,
$P$, $Q\in \mathcal{P}(\mathbb{C}^n)$, and $C$ a~constant
depending only on the degrees of $P$, $Q\in
\mathcal{P}(\mathbb{C}^n)$; see, e.g., \cite{BBEM, Duncan,
Mahler2}.

It was proved by Duncan \cite[Theorem 3]{Duncan} that if
$P,Q\in \mathcal{P}(\mathbb{C})$  with $m=
{\rm{deg}}(P)$ and $k= {\rm{deg}}(Q)$, then
\[
\|PQ\|_{L^2(\mathbb{T})} \geq {2m \choose m}^{-1/2} {2k \choose
k}^{-1/2}\,\|P\|_{L^2(\mathbb{T})}\,\|Q\|_{L^2(\mathbb{T})}.
\]
Below we present a~more general multidimensional variant which
estimates the  $L^p$-norms of products of polynomials over the
$n$-torus $\mathbb{T}^n$ also in the quasi-Banach case, i.e.,
$0<p<1$.

\vspace{2 mm}

\begin{proposition} \label{Duncan}
For every  $P,Q\in \mathcal{P}(\mathbb{C}^n)$ with $m=
{\rm{deg}}_\infty(P)$ and $k= {\rm{deg}}_\infty(Q)$, and every
$0<p<\infty$ we have
\begin{align*}
\|PQ\|_{L^p(\mathbb{T}^n)} \geq \big(\Lambda(p, m)\, \Lambda(p,
k
)\big)^{-n} \|P\|_{L^p(\mathbb{T}^n)}\,
\|Q\|_{L^p(\mathbb{T}^n)}\,.
\end{align*}

\vspace{1.5 mm} \noindent In particular, if $P(z)=
\sum_\alpha c_{\alpha}(P) z^{\alpha}$ and $Q(z)=
\sum_\alpha c_{\alpha}(Q) z^{\alpha}$, then
\begin{align*}
\bigg(\sum_\alpha
|c_{\alpha}(PQ)|^2\bigg)^{1/2} \geq {2m \choose m}^{-n/2} {2k
\choose k}^{-n/2}\, \bigg(\sum_\alpha
|c_{\alpha}(P)|^2\bigg)^{1/2}\,\bigg(\sum_\alpha |c_{\alpha}(Q)|^2\bigg)^{1/2}.
\end{align*}
\end{proposition}

\vspace{1.5 mm}

\begin{proof}
Combining Corollary \ref{m-polynomial} with multiplicativity
property \eqref{multiplicativity} of the Mahler measure $M$,
we obtain
\begin{align*}
\|P\|_{L^p(\mathbb{T}^n)} \|Q\|_{L^p(\mathbb{T}^n)} & \leq
\Lambda(p, m)^n M(P)\,\Lambda(p, k)^{n} M(Q) \\
& = \big(\Lambda(p, m)\,\Lambda(p, k)\big)^n M(PQ) \leq
\big(\Lambda(p, m)\,\Lambda(p, k)\big)^n
\|PQ\|_{L^p(\mathbb{T}^n)}.
\end{align*}
To conclude note that $$ \|R\|_{L^2(\mathbb{T}^n)} =
\Big(\sum_{\alpha} |c_{\alpha}(R)|^2\Big)^{1/2}\,,\quad R\in
\mathcal{P}(\mathbb{C}^n)\,,$$ and $\Lambda(2, m)  = {2m \choose
m}^{1/2}$ for each $m\in \mathbb{N}$.
\end{proof}

\vspace{2 mm}

In \cite{Mahler1} Mahler  proved the following univariate
``triangle inequality",
\begin{align*}
\label{madu} M(P + Q) \leq \kappa(m)\,\big(M(P) + M(Q)\big),
\quad\, P,\, Q \in \mathcal{P}(\mathbb{C})
\end{align*}
with the constant $\kappa(m)= 2^{m}$ where $\text{deg}(P) =
\text{deg}(Q)=m$, and observed that it has applications in the
theory of diophantine approximation. Duncan \cite{Duncan} obtained
the above inequality with the smaller constant
\[
\kappa(m) = {2m \choose m}^{1/2} \approx (\pi m)^{-1/4}\,2^m\,.
\]
We refer the reader to the paper \cite{Arestov1} of Arestov which
contains a~refinement of these results for algebraic polynomials
on the unit circle; more precisely, the following two-sided
estimates
\[
\frac{1}{2}r^{m} \leq \kappa(m) \leq \frac{1}{2}R^{m}, \quad\,
r\approx 1,796, \quad\, R=\sqrt[6]{40} \approx 1,8493
\]
are obtained, where $m\geq 6$, $R=\sqrt[6]{40} \approx 1,8493$ and
$r = \text{exp}(2G/\pi)r\approx 1,7916$ with $G=
\sum_{\nu=0}^{\infty} (-1)^{\nu}/(2\nu + 1)^{2}$.

\vspace{2 mm}

We have the following multidimensional variant of Duncan's result.

\vspace{2 mm}

\begin{proposition}
For $P$, $Q \in \mathcal{P}(\mathbb{C}^n)$ with
$m={\rm{deg}}_{\infty}(P) = {\rm{deg}}_{\infty}(Q)$ we have
\[
M(P + Q) \leq {2m \choose m}^{n/2}\big(M(P) + M(Q)\big).
\]
\end{proposition}

\vspace{1.5 mm}

\begin{proof} The proof is a direct consequence of Theorem
\ref{m-polynomial}: Since $\Lambda(2, m) = {2m \choose m}^{1/2}$
for each $m\in \mathbb{N}$, Jensen's inequality in combination
with Theorem \ref{m-polynomial} yields
\begin{align*}
M(P + Q) & \leq \|P + Q\|_{L^2(\mathbb{T}^n)} \leq
\|P\|_{L^2(\mathbb{T}^n)} + \|Q\|_{L^p(\mathbb{T}^n)}
 \leq {2m \choose m}^{n/2}\big(M(P) + M(Q)\big).
\end{align*}

\end{proof}

\vspace{2 mm}

We conclude the paper with  an interpolation inequality for
$L^p$-norms of polynomials on $\mathbb{T}^n$ that is interesting
its own; here for a~given measure space $(\Omega, \Sigma, \mu)$,
$E\in \Sigma$ with $\mu(E)>0$, and $0<p<\infty$, we write
$\|f\|_{L^p(E)}:= \big(\int_{E}|f|^p\,d\mu\big)^{1/p}\,,\,\,f\in
L^p(\mu)$.

\vspace{2 mm}

\begin{theorem}
Let $E$ be a~Lebesgue measurable subset of \,$\mathbb{T}^n$ with
$\lambda_{n}(E) \geq \theta>0$. Then the following interpolation
inequality holds for any polynomial $P\in
\mathcal{P}(\mathbb{C}^n)$ with ${\rm{deg}}_{\infty}(P)=m$ and any
$0<p < \infty${\rm{:}}
\[
\|P\|_{L^p(\mathbb{T}^n)} \leq
C(\theta)\,\Lambda(p,m)^{n}\,\big(\|P\|_{L^1(\mathbb{T}^n)}\big)^{1-\theta}\,
\big(\|P\|_{L^1(E)}\big)^{\theta},
\]
where $C(\theta) =
\frac{1}{\theta^{\theta}(1-\theta)^{1-\theta}}\leq 2$ for every
$0<\theta<1$, and $C(1)=1$.
\end{theorem}

\vspace{2.0 mm}

\begin{proof} The inequality is immediate consequence of Corollary
\ref{m-polynomial}  in combination with the following lemma which
is surely known to specialists -- for the sake of completeness we
include a~proof of it.
\end{proof}

\vspace{1.5 mm}

\begin{lemma}
Let $(\Omega, \Sigma, \mu)$ be a~probaility measure. If $\log
|f|\in L^1(\mu)$ and $A\in \Sigma$ with $\mu(A)=\alpha$, then
\[
{\rm{exp}} \int_{\Omega} \log |f|\,d\mu  \leq
\frac{1}{\alpha^{\alpha}(1-\alpha)^{1-\alpha}}\,
\bigg(\int_{A}|f|\,d\mu\bigg)^{\alpha}\,\bigg(\int_{\Omega\setminus
A} |f|\,d\mu\bigg)^{1-\alpha}.
\]
In particular, if $\mu(A)\geq \theta>0$, then
\[
{\rm{exp}} \int_{\Omega} \log |f|\,d\mu  \leq
2\,\bigg(\int_{\Omega} |f|\,d\mu\bigg)^{1-\theta}\,\bigg(\int_{A}
|f|\,d\mu\bigg)^{\theta}\,.
\]
\end{lemma}

\vspace{1.5 mm}

\begin{proof}
Recall that if $(\Omega, \Sigma, \nu)$ is a~probability measure
space and $\log|g| \in L^1(\nu)$, then it follows by Jensen's
inequality that
\begin{align*}
\label{J}
\text{exp} \int_{\Omega} \log |g|\,d\nu \leq \int_{\Omega} |g|\,
d\nu.
\end{align*}
With $A'= \Omega\setminus A$ this inequality yields
\begin{align*}
{\rm{exp}} \int_{\Omega} \log |f|\,d\mu & = \bigg(\text{exp}
\int_A \log |f|\,d\mu\bigg)\,\bigg(\text{exp}
\int_{A'} \log |f|\,d\mu\bigg)\\
&= \bigg (\text{exp} \frac{1}{\mu(A)} \int_A \log
|f|\,d\mu\bigg)^{\mu(A)} \bigg(\text{exp} \frac{1}{\mu(A')}
\int_{A'} \log
|f|\,d\mu\Big)^{\mu(A')} \\
& \leq \frac{1}{\theta^{\theta}(1-\theta)^{1-\theta}}
\bigg(\int_{A}|f|\,d\mu\bigg)^{\theta}\,\bigg(\int_{A'}
|f|\,d\mu\bigg)^{1-\theta}.
\end{align*}
To complete the proof it is enough to observe that $x^{x}(1-
x)^{1- x} \geq 1/2$ for every $0< x <1$.
\end{proof}

\vspace{2 mm}

{\bf Acknowledgements.}  We thank S.~Kwapie\'n for his permission
to include his result from Proposition \ref{homo3}. Finally, we
are grateful to D.~Galicer for some clarifying discussions.

\vspace {2 mm}

\noindent
Institut f\"ur Mathematik \\
Carl von Ossietzky Universit\"at \\
Postfach 2503 \\
D-26111 Oldenburg, Germany

\vspace{1 mm} \noindent
E-mail: andreas.defant@uni-oldenburg.de \\
\vspace{0.5 mm}

\noindent
Faculty of Mathematics and Computer Science\\
A.~Mickiewicz University\\
and Institute of Mathematics\\
Polish Academy of Sciences (Pozna\'n branch)\\
Umultowska 87\\
61-614 Pozna\'n, Poland

\vspace{1 mm} \noindent
E-mail: mastylo$@$amu.edu.pl \\
\end{document}